\documentclass[12pt,reqno]{amsart}
\usepackage{latexsym,amsmath}
\usepackage[left=3cm,top=2.5cm,right=3cm,bottom=2.5cm]{geometry}
\usepackage[usenames,dvipsnames]{color}
\usepackage{amsthm}
\usepackage{amssymb}
\usepackage{epsfig}
\usepackage{mathrsfs}
\usepackage{verbatim}
\usepackage[normalem]{ulem}
\DeclareMathOperator{\ex}{ex}

\def\I{\mathcal{I}}

\def\P{\mathcal{P}}

\def\R{\mathcal{R}}
\def\U{\mathcal{U}}
\def\V{\mathcal{V}}
\def\cV{\mathcal{V}}
\def\W{\mathcal{W}}

\newcommand{\hm}[1]{\leavevmode{\marginpar{\tiny%
$\hbox to 0mm{\hspace*{-0.5mm}$\leftarrow$\hss}%
\vcenter{\vrule depth 0.1mm height 0.1mm width \the\marginparwidth}%
\hbox to 0mm{\hss$\rightarrow$\hspace*{-0.5mm}}$\\\relax\raggedright #1}}}

\allowdisplaybreaks[2]
\newtheorem{theorem}{Theorem}[section]

\newtheorem{lemma}[theorem]{Lemma}

\newtheorem{claim}[theorem]{Claim}
\newtheorem{conjecture}[theorem]{Conjecture}
\newtheorem{definition}[theorem]{Definition}
\theoremstyle{definition}

\newcommand\eps{\varepsilon}
\def\({\left(}
\def\){\right)}

\def\L{\mathcal{L}}

\DeclareMathOperator{\Fano}{Fano}

\definecolor{red}{rgb}{1, 0, 0}
\def\epsilon{\varepsilon}

\begin{document}
\overfullrule=5pt

\title{The rainbow Erd\H{o}s-Rothschild problem for the Fano plane}

\author[L. Contiero]{Lucas de Oliveira Contiero}
\address{Instituto de Matem\'atica e Estat\'{i}stica, UFRGS -- Avenida Bento Gon\c{c}alves, 9500, 91501--970 Porto Alegre, RS, Brazil}
\email{lucas.contiero@ufrgs.br}

\author[C. Hoppen]{Carlos Hoppen}
\address{Instituto de Matem\'atica e Estat\'{i}stica, UFRGS -- Avenida Bento Gon\c{c}alves, 9500, 91501--970 Porto Alegre, RS, Brazil}
\email{choppen@ufrgs.br}

\author[H. Lefmann]{Hanno Lefmann}
\address{Fakult\"at f\"ur Informatik, Technische Universit\"at Chemnitz,
Stra\ss{}e der Nationen 62, D-09107 Chemnitz, Germany}
\email{Lefmann@Informatik.TU-Chemnitz.de}

\author[K. Odermann]{Knut Odermann}
\address{Fakult\"at f\"ur Informatik, Technische Universit\"at Chemnitz,
Stra\ss{}e der Nationen 62, D-09107 Chemnitz, Germany}
\email{knut.odermann@Informatik.TU-Chemnitz.de}

\thanks{This work was partially supported by CAPES and DAAD via Probral (CAPES Proc.~88881.143993/2017-01 and DAAD~57391132 and 57518130). The first author was supported by CAPES. The second author acknowledges the support of CNPq (Proc.~308054/2018-0).}


\begin{abstract}
The Fano plane is the unique linear 3-uniform hypergraph on seven vertices and seven hyperedges. It was recently proved that, for all $n \geq 8$, the balanced complete bipartite 3-uniform hypergraph on $n$ vertices, denoted by $B_n$, is the 3-uniform hypergraph on $n$ vertices with the largest number of hyperedges that does not contain a copy of the Fano plane. For sufficiently large $r$ and $n$, we show that $B_n$ admits the largest number of $r$-edge colorings with no rainbow copy of the Fano plane.
 \end{abstract}

\maketitle

\section{Introduction} 

This paper contributes to a line of research about coloring problems on combinatorial structures that originated from a graph-theoretical question of Erd\H{o}s and Rothschild~\cite{Erd74}. Their question was motivated by the Tur\'{a}n problem~\cite{turan}. 

As usual, for a fixed graph $F$, we say that a graph $G$ is \emph{$F$-free} if it does not contain $F$ as a subgraph. The \emph{Tur\'{a}n problem} associated with $F$ asks us to find the maximum number of edges among all $F$-free $n$-vertex graphs, which is denoted $\ex(n,F)$, and to determine the $F$-free $n$-vertex graphs $G$ with this number of edges, known as the \emph{$F$-extremal} graphs. 

Here we consider a related problem, which deals with $r$-colorings of the edge set of $G$ that do not contain a copy of $F$ colored \emph{according to a fixed pattern}. An \emph{$r$-edge-coloring} (or simply \emph{$r$-coloring}) $\Delta$ of a graph $G=(V,E)$ is a function $\Delta \colon E \rightarrow [r]$, where $[r]=\{1,\ldots,r\}$, and an \emph{$r$-pattern} $P$ of a graph $F$ is a partition of the edge set of $F$ into at most $r$ classes. An $r$-coloring of $G$ is said to be \emph{$(F,P)$-free} if it does not contain a copy of $F$ in which the partition of the edge set induced by the coloring is isomorphic to $P$. Given a graph $G$, one may consider the number $c_{r,(F,P)}(G)$ of $(F,P)$-free $r$-colorings of $G$ and define $c_{r,(F,P)}(n)$ as the maximum of this quantity over all $n$-vertex graphs. An $n$-vertex graph $G$ for which $c_{r,(F,P)}(G)=c_{r,(F,P)}(n)$ is said to be $(r,F,P)$-extremal.

The original Erd\H{o}s-Rothschild question concerned the instance where $F$ is a complete graph $K_\ell$ and the pattern $P$ consists of a single class, i.e., Erd\H{o}s and Rothschild were interested in edge-colorings of $G$ avoiding \emph{monochromatic} copies of $K_\ell$. Later, Pikhurko, Staden and Yilma~\cite{PSY2017} proposed a generalization of the original monochromatic problem, while Balogh~\cite{balogh} and two of the current authors~\cite{forbmx} studied problems leading to the more general version considered here. A non-monochromatic pattern that attracted considerable attention for various graphs $F$ is the \emph{rainbow} pattern, namely the pattern where all partition classes are singletons.

Since the late 90's, several researchers have obtained substantial progress both for the monochromatic case~\cite{ABKS04,han,PY12,Yuster} and for other patterns~\cite{BalLi18?,BBMS18?,gallai,linear,rainbow_triangle,rainbow_complete,genrainbow}. Typically, the sets of $F$-extremal graphs and $(r,F,P)$-extremal graphs coincide if $P$ is monochromatic and $r \leq 3$ or if $P$ is rainbow and $r$ is sufficiently large. Versions of the Erd\H{o}s-Rothschild problem have also been studied in the monochromatic setting for set systems~\cite{CDT2018,kneser}, linear spaces~\cite{CDT2018,linear_spaces}, partial orders~\cite{DGST} and sum-free sets~\cite{HJ2018}, for instance.

In this paper, we consider this problem for $k$-uniform hypergraphs, that is, for pairs $H = (V, E)$ where $V$ is a finite set called the \emph{vertex set} of $H$ and $E \subseteq \{e \colon e \subseteq V, |e|=k\}$ is called the \emph{edge set} of $H$. For a fixed $k$-uniform hypergraph $F$, the concepts of \emph{$F$-free hypergraph} and of \emph{$F$-extremal} 
$k$-uniform hypergraph with $n$ vertices may be defined as in the graph case. The same holds for the Tur\'{a}n number $\ex(n,F)$. For a positive integer $r$, an \emph{$r$-coloring} of a hypergraph $H = (V, E)$ is again a function associating  each hyperedge in $E$ with a color in $[r]$ and a \emph{color pattern} $P$ of a hypergraph $F$ is a partition of its edge set. An $r$-coloring of a hypergraph $H$ is said to be \emph{$(F, P)$-free} if there is no copy of $F$ in $H$ such that the partition of $E(F)$ induced by the $r$-coloring of $H$ is isomorphic to $P$. For fixed integers $n$, $r$ and $k$, if $\mathcal{H}_{n,k}$ denotes the set of $n$-vertex $k$-uniform hypergraphs, $F$ is a fixed $k$-uniform hypergraph and $P$ is a pattern of $F$, let $c_{r,(F,P)}(H)$ be the number of $(F,P)$-free $r$-colorings of $H$. The aim is to determine the function
\begin{equation}\label{eq_cr}
c_{r,(F,P)}(n)=\max\{c_{r,(F,P)}(H) \colon H \in \mathcal{H}_{n,k}\}
\end{equation}
and the hypergraphs $H \in \mathcal{H}_{n,k}$ that achieve equality in~\eqref{eq_cr}. The hypergraphs satisfying this are said to be \emph{$(r,F,P)$-extremal}. When $P$ is the monochromatic pattern, exact results about this problem have been obtained, for sufficiently large $n$, when $F$ lies in some class of expanded graphs~\cite{LP11} and when $F$ is the Fano plane~\cite{LPRS}. The \emph{Fano plane} is the unique linear $3$-uniform hypergraph on seven vertices and with seven hyperedges. Results for the monochromatic pattern in more general hypergraphs may be found in~\cite{LPS11}.

A proof method that has been particularly useful for the monochromatic pattern in instances where the $F$-extremal configuration is dense (that is, where $\ex(n,F)=\Omega(n^k)$) was developed by Alon, Balogh, Keevash and Sudakov~\cite{ABKS04}. Their approach consists of two steps:  (i) Prove a stability result establishing that any counterexample $H$ to the desired result would be similar to the actual $n$-vertex $F$-extremal structure. (ii) Assuming that there is an $n$-vertex graph other than the $F$-extremal graph with at least as many $r$-colorings as the extremal graph, prove that there is a sub-hypergraph whose number of colorings creates a `gap' to the number of colorings of the $F$-extremal configuration that is even larger. A recursive application of this step leads to a counterexample whose number of colorings is too large to be feasible.  

For instance, the results in~\cite{LPRS} are an adaptation of this approach to the case when $F$ is the Fano plane, $P$ is the monochromatic pattern and $r \in \{2,3\}$. As a second example, the results in~\cite{LPS11} give step (i) for arbitrary $F$, where $P$ is again the monochromatic pattern and $r \in \{2,3\}$. This is used in~\cite{LP11} to derive the exact result (namely step (ii)) when $F$ is an expanded complete graph or a Fan hypergraph.

We show that, for the rainbow pattern $P$ of the $3$-uniform Fano plane and sufficiently large $r$ and $n$, the \emph{stability} given in (i) holds. We then use the strategy in (ii) to prove that this stability implies that the $\Fano$-extremal and the $(r,\Fano,P)$-extremal $n$-vertex configurations coincide for large $n$.

 We finish the introduction with the precise statement of our main result. In the remainder of this paper, we shall write $F^R$ for the rainbow pattern of the Fano plane. In particular, we shall refer to $F^R$-free $r$-colorings, to $(r,F^R)$-extremal hypergraphs and to the function $c_{r,F^R}(n)$ instead of $(\Fano,P)$-free $r$-colorings, $(r,\Fano,P)$-extremal hypergraphs and $c_{r,(\Fano,P)}(n)$, respectively.
 
 Given a positive integer $n$, let $B_n$ be the $n$-vertex 3-uniform hypergraph defined as follows. There is a partition $V(B_n) = V_1\cup V_2$ of the vertex set of $B_n$, with $||V_1|-|V_2||\leq 1$, such that $E(B_n)$ consists of all triples having non-empty intersection with both classes. F\"uredi and Simonovits~\cite{FSstability} and, independently,  Keevash and Sudakov~\cite{KSstability} proved that $B_n$ is the unique $F$-extremal hypergraph for all $n$ sufficiently large. Recently, Bellmann and Reiher~\cite{BRstability} proved that this holds for every $n \geq 8$.

We prove that $B_n$ is $(r,F^R)$-extremal if $r$ and $n$ are sufficiently large.
\begin{theorem}\label{thm:main:fano} 
The following holds for every 
$$r \geq r_0= 6^{492^{64}\cdot (37\cdot 16^3\cdot 1406^9)^{63}}. 
$$
 There exists $n_0=n_0(r)$ such that for any $n \geq n_0$ and any 3-uniform $n$-vertex hypergraph $H$, the inequality
\begin{equation}\label{eq_fano}
c_{r,F^R}(H) \leq r^{\ex(n,\Fano)}
\end{equation}
holds. Moreover, for $n \geq n_0$, equality holds in~\eqref{eq_fano} if and only if $H$ is isomorphic to $B_n$.
\end{theorem}

Our paper is structured as follows. In Section~\ref{sec:aux} we introduce definitions and auxiliary results that will be useful for proving our main theorem. In Section~\ref{sec:stability} we will derive a colored stability result, which is Theorem~\ref{thm:kee1}.  With this we prove an embedding result in Section~\ref{sec:embedding}, which implies Theorem~\ref{thm:main:fano}. We conclude the paper with final remarks and open problems.

\section{Notation and tools} \label{sec:aux}

In this section, we state auxiliary results that will be used in our proofs. In addition to fairly standard definitions and technical results, we shall consider a version of the regularity method for hypergraphs, known as the Weak Hypergraph Regularity Lemma, and embedding results related to it.

\subsection{Regularity Lemma}
A key tool in this paper is the so-called \textrm{weak hypergraph regularity lemma}. It is a natural extension of the Szemer\'{e}di Regularity Lemma \cite{Sregularity}. 

Let $H = (V, E)$ be a $k$-uniform hypergraph and let $W_1, \ldots, W_k$ be mutually disjoint non-empty subsets of $V$. Let $E(W_1, \ldots, W_k) = \{e\in E(H)\colon |e\cap W_i|=1 \;  \forall i\in[k]\}$ and consider the \emph{density}
$$d_H(W_1, \ldots, W_k) = \frac{|E(W_1, \ldots, W_k)|}{|W_1|\cdots|W_k|}$$
of $H$ with respect to the sets  $W_1, \ldots, W_k$. 

For fixed $\eps > 0$ and $d>0$, we say that $(V_1, \ldots, V_k)$ 
is \emph{$(\eps,d)$-regular},  if
$$
|d_H(W_1, \ldots, W_k)-d|\leq \eps$$
for all $k$-tuples $(W_1, \ldots , W_k)$ of 
subsets $W_i\subseteq V_i$, $i\in [k]$,  
with $\prod_{i=1}^k|W_i| \geq \eps 
\prod_{i=1}^k |V_i|$. Such an $k$-tuple $(V_1,\ldots,V_k)$ is said to be \emph{$\eps$-regular} if it is $(\eps,d)$-regular for some $d\geq 0$, and \emph{$\eps$-irregular} otherwise.

Finally, an equitable partition $\V=\{V_1,\ldots,V_t\}$\footnote{This is a partition that satisfies $||V_i|-|V_j||\leq 1$ for all $i,j \in [t]$.} of the vertex set of a $k$-uniform hypergraph $H=(V,E)$ is \emph{$\eps$-regular} if, for all but at most $\eps
 \binom{t}{k}$ distinct $k$-element subsets
    $\{i_1, \ldots, i_k \}\subseteq[t]$, the $k$-tuple
    $(V_{i_1}, \ldots, V_{i_k})$ is
  $\eps$-regular.

The following is a colored version of the Weak Regularity Lemma that will be used in this paper, which may be easily derived from the work in \cite{chung91,fr92,KNRSembedding,steger90}.
\begin{theorem}[Colored Regularity Lemma]\label{theoremcoloredregularity}
For all integers $r\geq 1$, $k\geq 2$ and $m_0\geq 1$, and every $\epsilon>0$, there exist $M_0 = M_0(r, k,m_0, \epsilon)$ and $N_0 = N_0(r,k,m_0, \epsilon)$ with the following property. Every $k$-uniform hypergraph $H = (V, E)$ on $n\geq N_0$ vertices whose set of hyperedges is $r$-colored $E(H) = E_1\cup\cdots\cup E_r$ admits an equitable partition $\V = \{V_1,\ldots, V_m\}$ of $V$ with $m_0\leq m\leq M_0$ that is simultaneously $\epsilon$-regular for all sub-hypergraphs $H_i = (V, E_i)$, where $i\in [r]$.
\end{theorem}

This colored regularity lemma gives rise to a cluster hypergraph where, for each regular $k$-tuple, we record the colors that appear with density larger than a certain threshold $\eta$.
\begin{definition}[Multicolored cluster hypergraph]\label{defmulticolcluster}
Let $H = (V, E)$ be a $k$-uniform hypergraph, whose set of hyperedges is $r$-colored $E(H) = E_1\cup\cdots\cup E_r$. Consider an  equitable partition $\V = \{V_1,\ldots, V_m\}$ of $V$ that is simultaneously $\epsilon$-regular for all sub-hypergraphs $H_i = (V, E_i)$, where $i\in[r]$.  For $\eta>0$, the \emph{multicolored cluster hypergraph} $\R = \R_H(\V, \eta)$ associated with this partition $\V$ and $\eta$ is the hypergraph with vertex set $[m]$ and edge set $E_{\mathcal{R}}$ defined as follows. A $k$-set $\{i_1, \ldots, i_k\}\in\binom{[m]}{k}$ belongs to $E_{\mathcal{R}}$ if the $k$-tuple $(V_{i_1}, \ldots, V_{i_k})$ is $\epsilon$-regular with respect to all colors and $d_{H_i}(\{V_{i_1}, \ldots, V_{i_k}\})\geq \eta$ for some $i\in[r]$.  Moreover, every hyperedge $\{{i_1}, \ldots, {i_k}\}\in E_\R$ is assigned a list $L_{\{{i_1}, \ldots, {i_k}\}} = \{i\in[r]\colon d_{H_i}(\{V_{i_1}, \ldots, V_{i_k}\})\geq\eta\}$ of \emph{$\eta$-dense} colors with respect to $(V_{i_1}, \ldots, V_{i_k})$.
\end{definition}

A counting lemma  (also known as the Key Lemma) in the context of the Weak Hypergraph Regularity Lemma was proved in~\cite{KNRSembedding}. However, it holds only for \emph{linear hypergraphs}, that is, for hypergraphs $F$ for which the size of the intersection of any two hyperedges is at most one. Here we state this result in a colored form. We say that a hypergraph $F$ is a \emph{colored hypergraph} if every hyperedge is assigned a color and $\R$ is a \emph{multicolored hypergraph} if every hyperedge of $\R$ is assigned a nonempty list of colors.
\begin{definition}[Colored sub-hypergraph]
Let $\R$ be a multicolored hypergraph. A colored hypergraph $F$ is said to be a \emph{colored sub-hypergraph} of $\R$ if the following holds:
\begin{itemize} 
\item[(a)] Ignoring colors and lists of colors, $F$ is a sub-hypergraph of $\R$. 
\item[(b)] For every hyperedge $e\in E(F)$, the color of $e$ in $F$ belongs to the list $L_e$ of colors of $e$ in $\R$.
\end{itemize}
\end{definition}

\begin{lemma}[Multicolored Embedding Lemma] \label{lem:clusterlemma}
For all integers $r, k\geq 2$, and every $\eta>0$ there exist a constant $\epsilon=\epsilon(r, k, \eta)$ and an integer $s_0 = s_0(r, k, \eta)$ such that the following holds for every positive integer $m$. Let $H = (V, E)$ be a hypergraph whose hyperedges are $r$-colored $E(H) = E_1\cup\cdots\cup E_r$, and consider a partition $\V = \{V_1,\ldots, V_m\}$ of $V$ that is simultaneously $\epsilon$-regular for all sub-hypergraphs $H_i = (V, E_i)$, where $i\in[r]$. Further assume that $|V_j|\geq s_0$ for every $j\in[m]$. Fix a colored linear hypergraph $F$. If $F$ is a colored sub-hypergraph of $\R_H(\V, \eta)$, then $H$ contains a copy of $F$.
\end{lemma}

\subsection{The Fano plane} The \emph{Fano plane} is the unique linear $3$-uniform hypergraph on seven vertices and seven hyperedges. As mentioned in the introduction, F\"uredi and Simonovits~\cite{FSstability} and, independently, Keevash and Sudakov~\cite{KSstability} have proved that the unique extremal hypergraph for the Fano plane is $B_n$, for $n$ sufficiently large. Bellmann and Reiher~\cite{BRstability} proved this for all $n \geq 8$, which is best possible. Note that, considering the parity of $n$, we have 
\begin{eqnarray}\label{eq|E(B_n)|}
\frac{n^3}{8}-\frac{n^2}{4}-\frac{n}{8}+\frac{1}{4} \leq |E(B_n)| = \ex(n, \Fano) \leq \frac{n^3}{8} - \frac{n^2}{4}.
\end{eqnarray}

For a $3$-uniform hypergraph $H=(V,E)$ and a subset $A \subseteq V$, let $E_H(A)$ denote the set and $e_H(A)$ denote the number of hyperedges that are contained in $A$ (we write $E(A)$ or $e(A)$ if the hypergraph under consideration is obvious).

To derive the extremality of $B_n$, Keevash and Sudakov~\cite{KSstability} (and F\"{u}redi and Simonovits~\cite{FSstability}) established a stability result as follows.
\begin{theorem}[Stability]\label{theoremstabilityfano}
For every $\delta>0$ there exist $\epsilon = \epsilon(\delta)>0$ and $n_0$ such that every $n$-vertex $\Fano$-free hypergraph $H = (V, E)$ with $n>n_0$ containing at least $\ex(n,\Fano)-\epsilon n^3$ hyperedges admits a partition $V(H) = A \cup B $ with $e(A)+e(B)\leq \delta n^3$.
\end{theorem}

A careful analysis of the proof of~\cite[Theorem~1.2]{KSstability} allows us to obtain the following quantitative version of Theorem~\ref{theoremstabilityfano}, whose proof may be found in~\cite{stabilityFano}.
\begin{theorem}\label{thm:kee1}
For any fixed $0 < \delta \leq  1/36^8$, there exists $n_0$ such that the following holds for all $n \geq n_0$. If $H = (V,E)$ is a $\Fano$-free $3$-uniform hypergraph on $n \geq n_0$ vertices with $\ex(n, \Fano)- \delta n^3$ hyperedges, then there is a partition $V(H) = A \cup B$ so that $e_H(A) + e_H(B)  < 2 \delta^{1/64}n^3$.
\end{theorem}

Let $H$ be a hypergraph and let $X \cup Y$ be a bipartition of its vertex set. Hyperedges of $H$ that contain at least one vertex from each of the two classes are said to be \emph{crossing}, and the set of all crossing hyperedges with respect to the given partition is denoted by $E_C(H)$. The set of all non-crossing hyperedges with respect to the given partition is denoted by $E_N(H)$
\begin{lemma}\label{lem:sizes}
For every $\delta > 0$ and
integer $n\geq \max \{8,1/\sqrt{\delta} \}$,  let $H = (V, E)$ be an $n$-vertex $3$-uniform  hypergraph such that there exists a partition $\V = \{X, Y\}$ of $V$, for which
$$|E_C(H)| \geq |E(B_n)|-\delta n^3.$$
The following inequalities hold:
$$\frac{n}{2}-2n\sqrt{\delta}\leq \min\{|X|, |Y|\}\leq\max\{|X|, |Y|\}\leq \frac{n}{2}+2n\sqrt{\delta}.$$
\end{lemma}

\begin{proof}
Let $|X|=a$ and $|Y| = n-a$. With~(\ref{eq|E(B_n)|}), we infer that
\begin{eqnarray*}
&&a\binom{n-a}{2} + (n-a)\binom{a}{2} \geq \ex(n, \Fano) - \delta n^3 \geq \frac{n^3}{8} - \frac{n^2}{4} - \frac{n}{8} + \frac{1}{4} - \delta n^3\\
&\Rightarrow&
an^2-2an-a^2n+2a^2\geq\frac{n^3}{4} - \frac{n^2}{2} - \frac{n}{4} - 2\delta n^3\\
&\Rightarrow& \left| a - \frac{n}{2} \right|  \leq \sqrt{\frac{n}{4(n-2)} + \frac{2\delta n^3}{n-2}} \stackrel{(n \geq 1/\sqrt{\delta}, n \geq 8)}{\leq} 2\sqrt{\delta} n. 
\end{eqnarray*}
\end{proof}

We will  use the entropy function $h \colon [0,1] \rightarrow [0,1]$ given by $h(x) = -x \log_2 x - (1-x) \log_2(1-x)$ for $0<x<1$ and $h(0) = h(1) = 0$. For  $0 \leq \alpha \leq 1$ the following inequality is well-known: 
\begin{eqnarray} \label{eq:entropy1}
\binom{n}{\alpha n} \leq 2^{h(\alpha) n}.
\end{eqnarray}

We will also use the following upper bound on the entropy function for $x \leq 1/8$:
\begin{eqnarray} \label{eq:entropy2}
h(x) &\leq& -2x \log_2 x.
\end{eqnarray}
 Namely, (\ref{eq:entropy2}) is equivalent to 
 $g(x) = x \ln x - (1-x) \ln (1-x) \leq 0$. Taking the derivative 
gives
$g'(x) = \ln x + 2 + \ln (1-x) \leq 0$ for $x \leq 1/8$. With  $g(1/8) < 0$ 
inequality~(\ref{eq:entropy2}) follows.

\section{A stability result for $F^R$-free $r$-colorings} \label{sec:stability}

To prove Theorem~\ref{thm:main:fano}, we first establish a stability result for colorings, which implies that, for $n$ sufficiently large,  any $n$-vertex hypergraph with a large number of $F^R$-free $r$-colorings must be structurally similar to the hypergraph $B_n$.

\begin{lemma}\label{lem:est:col:FK_fano}
 For any fixed $0 < \delta \leq 1/36^8$, there exists $r_0 = r_0(\delta)  = 6^{{492}^{64}/\delta^{63}} $ such that the following holds for all $r \geq r_0$. There is $n_1$ such that, if $n \geq n_1$ and $H = (V, E)$ is a $3$-uniform $n$-vertex hypergraph satisfying $c_{r,F^R}(H) \geq r^{\ex(n,\Fano)}$, then there is a partition $\cV = \{V_1,  V_2\}$ of $V$ with $e(V_1) + e(V_2) \leq \delta n^3$.
\end{lemma}

\begin{proof}
 Let $0 < \delta \leq 1/36^8$ be given. Let $X=1/432^{64}<1/36^8$. Define
\begin{eqnarray} \label{eq:star7} r_0 = \max\left\{6^{\frac{1}{12X}+2},  6^{\frac{492^{64}}{\delta^{63}}}  \right\}.
\end{eqnarray}
In fact, it is easy to see that the second of the above terms is larger, but we write it in this way to avoid distracting calculations later in the proof. Let $r\geq r_0$. Fix a positive number $\eta$ with
\begin{eqnarray}\label{eqLeta}
\eta < \frac{\delta}{4r} 
\end{eqnarray}
that satisfies
\begin{eqnarray}\label{eqLeta2}
\frac{\delta}{2} \leq 243  (4h(r\eta) + 4r\eta)^{\frac{1}{64}}< \frac{3\delta}{4}.
\end{eqnarray}
For this value of $\eta$, Lemma~\ref{lem:clusterlemma} gives us constants $\eps = \eps(r,3,\eta)$ and $s_0 = s_0(r,3,\eta)$ which we further assume to satisfy
\begin{eqnarray}\label{eqLepsilon}
\eps<\min\{3\eta/2, r\eta/4\}.
\end{eqnarray}
Fix $m_0$ with
\begin{eqnarray}\label{eqLt_0}
m_0\geq\frac{1}{\eps}.
\end{eqnarray}
Moreover, we shall also consider that $m_0$ is sufficiently large to ensure that some of the upcoming equations are satisfied (all such equations will be marked by $m_0\gg1$).

Let $N_0 = N_0(r, 3, m_0, \eps)$ and $M_0 = M_0(r, 3, m_0, \eps)$ be given by Theorem~\ref{theoremcoloredregularity}, and fix $n_1 \geq \max\{N_0, s_0  M_0\}$, such that the equations marked by $n \gg 1$ are satisfied for $n \geq n_1$.

For $F$ being the Fano plane, let $H$ be an $n$-vertex $3$-uniform hypergraph satisfying $ c_{r, F^R}(H)\geq r^{\ex (n, \Fano)}$ and fix one of its $
F^R$-free $r$-colorings. By Theorem~\ref{theoremcoloredregularity} there exists an equitable partition $\V = \{V_1,\ldots,V_m\}$ of $V$, where $m_0\leq m\leq M_0$, that is $\eps$-regular simultaneously with respect to each subhypergraph $H_i=(V,E_i)$, for all $i\in [r]$. Let $\mathcal{R} = \mathcal{R}_H(\mathcal{V},\eta)$
be the multicolored cluster hypergraph with vertex set $[m]$ associated with this partition. For a hyperedge $e$ in $\mathcal{R}$, let $L_e$ be its list of colors, as in Definition~\ref{defmulticolcluster}. Given a subset $\{i_1, i_2, i_3\}$ of $[m]$, we say that a hyperedge $e$ in $\mathcal{R}$ \emph{lies in} a triple $(V_{i_1},  V_{i_2}, V_{i_3})$ if $|e \cap V_{i_j}|=1$ for all $j \in [3]$.

Our aim is to find an upper bound on the number of $F^R$-free $r$-colorings of $H$. To do this, we sum over all possible $\eps$-regular partitions $\cV=\{ V_1, \ldots, V_m\}$ and all multicolored cluster hypergraphs $\mathcal{R}$, and we find an upper bound on the number of $r$-colorings (of sub-hypergraphs of $H$) for which $\cV$ is an $\eps$-regular partition (with respect to all colors) associated with the multicolored cluster hypergraph $\mathcal{R}$. 

Fix $\cV$ and $\mathcal{R}$. In the following, we assume that $m$ divides $n$ to avoid technicalities, but the same conclusion holds for other values of $m$ and $n$. There are at most $r \eps\binom{m}{3}$ subsets $\{i_1, i_2, i_3\}$ of $[m]$ for which the triple $(V_{i_1},  V_{i_2}, V_{i_3})$ is not $\eps$-regular with respect to the partition $\cV=\{ V_1, \ldots, V_m\}$ for at least one of the colors. At most
\begin{eqnarray}\label{eqL1}
r\eps\binom{m}{3}\left(\frac{n}{m}\right)^3 \leq r\eps\frac{n^3}{6}
\end{eqnarray}
hyperedges lie in a triple of this type. Because (\ref{eqLt_0}) holds, there are at most
\begin{eqnarray}\label{eqL2}
m\left(\frac{n}{m}\right)^2n\leq\eps n^3
\end{eqnarray}
hyperedges with at least two elements in a same class with respect to $\cV$. Next, consider hyperedges $e$ that lie in an $\eps$-regular triple $(V_{i_1}, V_{i_2} , V_{i_3})$ for which the set of hyperedges with the same color as $e$ has density less than $\eta$ with respect to $(V_{i_1},  V_{i_2}, V_{i_3})$. The number of such hyperedges is bounded above by
\begin{eqnarray}\label{eqL3}
r\eta\binom{m}{3}\left(\frac{n}{m}\right)^3 \leq r\eta\frac{n^3}{6}.
\end{eqnarray}

Using (\ref{eqL1}), (\ref{eqL2}), (\ref{eqL3}) and (\ref{eqLepsilon}),
we conclude that there are at most
$$
\left(\frac{r\eps}{6}+\eps+\frac{r\eta}{6}\right)n^3
\stackrel{(\ref{eqLepsilon})}{<}
r\eta n^3
$$
hyperedges in $H$ that do not lie in an $\eps$-regular triple $(V_{i_1},  V_{i_2}, V_{i_3})$ for which their color is dense. There are at most $\binom{n^3}{r \eta n^3}$ ways to fix these hyperedges in $H$, and they can be colored in at most $r^{r\eta n^3}$ ways. The remaining hyperedges of $H$ can be colored in at most
$$\left(\prod_{e\in E(\mathcal{R})}|L_e|\right)^{\left(\frac{n}{m}\right)^3}$$
ways. Thus, the total number of $r$-colorings of $H$ that give rise to the partition $\cV=\{ V_1, \ldots, V_m\}$ and the multicolored cluster hypergraph $\mathcal{R}$ is bounded above by
$$
\binom{n^3}{r\eta n^3}r^{r\eta n^3} \cdot
\left(\prod_{e\in E(\mathcal{R})}|L_e|\right)^{\left(\frac{n}{m}\right)^3}.
$$

Let $e_j(\mathcal{R}) = |\{e\in \mathcal{R}:|L_e|=j\}|$, $j \in [r]$. We write a sum over a set $\P$  of pairs  to denote a sum over all equitable partitions $\cV$ and all possible multicolored cluster hypergraphs $\mathcal{R}$ associated with $\cV$. We have
\begin{eqnarray}
c_{r, F^R}(H)
&\leq&
\sum_{(\cV,\mathcal{R}) \in  \P}
\binom{n^3}{r\eta n^3} r^{r\eta n^3}
\left(\prod_{e\in E(\mathcal{R})}|L_e|\right)^{\left(\frac{n}{m}\right)^3}\nonumber\\
&\stackrel{(\ref{eq:entropy1})}{\leq}&
\sum_{(\cV,\mathcal{R}) \in \P}
2^{h(r\eta)n^3}r^{r\eta n^3}
\left(\prod_{j=1}^{6}j^{e_j(\mathcal{R})}\prod_{j=7}^{r}j^{e_j(\mathcal{R})}\right)^{\left(\frac{n}{m}\right)^3}\nonumber\\
&\leq&
\sum_{(\cV,\mathcal{R}) \in \P}
2^{h(r\eta)n^3}r^{r\eta n^3}
\left(
6^{\sum_{j=1}^{6}e_j(\mathcal{R})}
r^{\sum_{j=7}^{r}e_j(\mathcal{R})}\right)^{\left(\frac{n}{m}\right)^3}. \label{eqLL}
\end{eqnarray}

Assume that there is a copy of the Fano plane $F$ in $\mathcal{R}$ so that every hyperedge $e$ in this copy satisfies $|L_e|\geq 7$. Then we may greedily assign a color from each list to produce a rainbow coloring of $F$, which would lead to a rainbow copy of $F$ in the original coloring because of Lemma~\ref{lem:clusterlemma}. As there is no such copy, we must have
$$
\sum_{j=7}^{r}e_j(\mathcal{R}) \leq \ex(m,\Fano).
$$

For a multicolored cluster hypergraph $\mathcal{R}=\mathcal{R}(\eta)$, let
$$\beta(\mathcal{R})=\frac{1}{m^3} \left(\ex(m,\Fano) -  \sum_{j=7}^{r}e_j(\mathcal{R})\right).$$

We consider two cases.
\begin{itemize}
\item[(a)] For every multicolored cluster hypergraph $\mathcal{R}$, we have $\beta(\mathcal{R}) \geq 4h(r\eta) + 4r\eta$.


\item[(b)] There exists a multicolored cluster hypergraph $\mathcal{R}$ for which $\beta(\mathcal{R}) < 4h(r\eta) + 4r\eta$.
\end{itemize}

\vspace{3pt}

\noindent \emph{Case (a)} Our aim is to show that this case cannot apply by proving that the number of $F^R$-free $r$-colorings of $H$ is less than $r^{\ex(n,\Fano)}$. We first consider all those terms $(\cV,\mathcal{R}) \in \P$ in the sum (\ref{eqLL}) where the multicolored cluster hypergraph 
$\mathcal{R}$ satisfies $\beta(\mathcal{R}) \geq X$. We denote the set of all these terms by $\P^\ast$. This leads to the upper bound
\begin{eqnarray}\label{eq:casea}
&&
\sum_{(\cV,\mathcal{R}) \in\P^\ast}
2^{h(r\eta)n^3}r^{r\eta n^3}
\left(
6^{\binom{m}{3}-(\ex(m,\Fano)-\beta(\mathcal{R}) m^3)}
r^{\ex(m,\Fano)-\beta(\mathcal{R}) m^3}\right)^{\left(\frac{n}{m}\right)^3}
\nonumber\\
&\leq&
\sum_{(\cV,\mathcal{R}) \in\P^\ast}
2^{h(r\eta)n^3}r^{r\eta n^3}
\left(
\frac{
6^{\frac{m^3}{24} +\beta(\mathcal{R}) m^3}
}
{
r^{\frac{\beta(\mathcal{R})}{2} m^3}
}
r^{\ex(m,\Fano)-\frac{\beta(\mathcal{R})}{2} m^3}
\right)^{\left(\frac{n}{m}\right)^3}.
\end{eqnarray}
However, as $r \geq r_0 \geq 6^{1/{(12X)}+2}$ by (\ref{eq:star7}), we have
\begin{eqnarray}
&&
\left(
\frac{
6^{\frac{1}{24}+\beta(\mathcal{R})}
}
{
r^{\frac{\beta(\mathcal{R})}{2}}
}
\right)^{m^3}
\leq
\left(
\frac{
6^{\frac{1}{24}+\beta(\mathcal{R})}
}
{
\left(6^{\frac{1}{12X}+2}\right)^{\frac{\beta(\mathcal{R})}{2}}
}
\right)^{m^3}
=
\left(
\frac{
6^{\frac{1}{24}}
}
{
6^{\frac{\beta(\mathcal{R})}{24X}}
}
\cdot \frac
{
6^{\beta(\mathcal{R})}
}
{
6^{\beta(\mathcal{R})}
}
\right)^{m^3}
\nonumber
\stackrel{(\beta(\mathcal{R})\geq X)}{\leq}
1.
\nonumber
\end{eqnarray}
Then, (\ref{eq:casea}) is at most
\begin{eqnarray}\label{eq:caseaa}
&&\sum_{(\cV,\mathcal{R}) \in\P^\ast}
2^{h(r\eta)n^3}r^{r\eta n^3}
r^{\ex(n,\Fano)-\frac{\beta(\mathcal{R})}{2} n^3}.
\end{eqnarray}

Since the number of classes is $m\leq M_0$, we have at most $M_0^n$ partitions $\cV=\{ V_1, \ldots, V_m\}$ and at most $2^{rM_0^3}$ multicolored cluster hypergraphs. Moreover, as we are in  case~(a) where $\beta(\mathcal{R}) \geq 4h(r\eta) + 4r\eta$,   the expression in~\eqref{eq:caseaa} is at most
\begin{eqnarray}\label{eq:casea1}
\nonumber M_0^n \cdot 2^{rM_0^3} \cdot
2^{h(r\eta)n^3} \cdot r^{r\eta n^3} 
r^{\ex(n,\Fano)- 2h(r\eta)n^3 -2r\eta n^3}
&\stackrel{n \gg 1}{\leq}&
 r^{-(h(r\eta) + r\eta) n^3} r^{\ex(n,\Fano)} \\
&<& \frac{1}{2} r^{\ex(n,\Fano)}.
\end{eqnarray}

Now we consider those terms $(\cV,\mathcal{R})$ in the sum (\ref{eqLL}) for which 
$\mathcal{R}$ satisfies $4h(r\eta) + 4r\eta \leq \beta(\mathcal{R}) \leq X$.
Fix one such pair $(\cV,\mathcal{R})$ and let $\mathcal{R}'$ be the hypergraph obtained from $\mathcal{R}$ by deleting all hyperedges $e$ satisfying $|L_e|\leq 6$. Since
$
(\ex(m,\Fano) - \sum_{j=7}^{r}e_j(\mathcal{R}))
=
\beta(\mathcal{R}) m^3$, and $\beta (\mathcal{R})\leq X=1/432^{64} <1/36^8$, 
and since there is no copy of $F^R$ in $\mathcal{R}'$, Theorem~\ref{thm:kee1} produces a partition $U_1\cup U_2$ of $V(\R')=[m]$ with respect to which 
\begin{eqnarray}\label{eq:caseb}
 e_{\R'}(U_1) +  e_{\R'}(U_2)
&\leq&
 2\beta(\mathcal{R})^{\frac{1}{64}}  m^3.
\end{eqnarray}

Let $\mathcal{R}''$ be the  subhypergraph of $\mathcal{R}'$ obtained by removing all hyperedges  in  $ E_{\R'}(U_1) \cup  E_{\R'}(U_2)$,  thus
\begin{eqnarray} \label{eq:101}
e(\mathcal{R}'') \geq \ex(m,\Fano) - \beta(\mathcal{R})m^3 - 
 2\beta(\mathcal{R})^{\frac{1}{64}}  m^3.
 \end{eqnarray}
Let $K=K(U_1,U_2)$ be the complete bipartite $3$-uniform hypergraph with partition $[m]= U_1\cup U_2$. 

For every hyperedge $f \in  E_{\R}(U_1) \cup  E_{\R}(U_2)$,  let $\mathcal{C}(f)$ be the set of all copies of a Fano plane $F$ in $K+f$. Let $C = \min\{|\mathcal{C}(f)|: f\in E_{\R}(U_1) \cup  E_{\R}(U_2)\}$. 

On the one hand we have
\begin{equation}\label{eq:lb}
C \cdot  | E_{\R}(U_1) \cup  E_{\R}(U_2)| \leq\sum_{f \in  E_{\R}(U_1) \cup  E_{\R}(U_2)} \sum_{F \in \mathcal{C}(f)} 1. 
\end{equation}
On the other hand, consider $\overline{E} = E(K)\setminus E(\mathcal{R}'')$ which satisfies 
\begin{eqnarray}\label{bound_over_E}
|\overline{E}| \leq |E(B_{m})| - |E(\mathcal{R}'')| &\stackrel{(\ref{eq:101})}{\leq}& \ex(m,\Fano)-\left(\ex(m,\Fano)-\beta(\mathcal{R}) m^3-  2\beta(\mathcal{R})^{\frac{1}{64}} m^3 \right)  \nonumber \\
&=& \beta(\mathcal{R}) m^3 +  2\beta(\mathcal{R})^{\frac{1}{64}} m^3  <   3 \beta(\mathcal{R})^{\frac{1}{64}} m^3.
\end{eqnarray} 
We claim that for every hyperedge $f\in  E_{\R}(U_1) \cup  E_{\R}(U_2)$ and $F \in \mathcal{C}(f)$, the Fano plane $F$ contains a hyperedge in $\overline{E}$. Assuming this, and using the fact that every hyperedge in $\overline{E}$ belongs to at most $m^{4}$ copies of $F$ in the complete $k$-uniform hypergraph with vertex set $[m]$, we obtain
\begin{equation}\label{eq:ub}
\sum_{f \in  E_{\R}(U_1) \cup  E_{\R}(U_2)} \sum_{F \in \mathcal{C}(f)} 1 \leq \sum_{g \in \overline{E}} \sum_{\stackrel{F\in \bigcup \mathcal{C}(f)}{g \in F}} 1
\leq m^{4} |\overline{E}|  \stackrel{(\ref{bound_over_E})}{<}   3\beta(\mathcal{R})^{\frac{1}{64}}   m^7.
\end{equation}
In this equation, $F\in \bigcup \mathcal{C}(f)$ denotes $F \in \bigcup_{f \in  E_{\R}(U_1) \cup  E_{\R}(U_2)} \mathcal{C}(f)$. Combining~\eqref{eq:ub} with \eqref{eq:lb}, we derive
 \begin{eqnarray}\label{eq:aux}
 e_{\R}(U_1) +  e_{\R}(U_2) 
\leq \frac{   3\beta(\mathcal{R})^{\frac{1}{64}}   m^7}{C}.
\end{eqnarray}

Before finding a lower bound on $C$, we show that, for every $f\in E_{\R}(U_1) \cup  E_{\R}(U_2)$ and $F \in \mathcal{C}(f)$, the Fano plane $F$ indeed contains a hyperedge in $\overline{E}$. Given a hyperedge $f\in E_{\R}(U_1) \cup  E_{\R}(U_2)$ and $F \in \mathcal{C}(f)$, we know that $E(F)\setminus\{f\}\subseteq E(K)$. Moreover, because $f \in E(\R)$, we know that its list has at least one color. From this we derive that $E(F)\setminus\{f\}\not\subseteq E(\R'')$, as otherwise the hyperedges of $E(F) \setminus \{f\}$ would have at least seven colors in their list of colors, which would lead to a rainbow copy of $F$ in $\R$ where every hyperedge is assigned a color of its own list. Lemma~\ref{lem:clusterlemma} produces the desired contradiction.

\begin{claim}\label{claim:copies}
Given a hyperedge $f\in E_{\R}(U_1) \cup  E_{\R}(U_2)$, we have 
$
|\mathcal{C}(f)|
\geq
\left(m/3\right)^{4}.
$
\end{claim}

\begin{proof}
Given a hyperedge $f \in E_{\R}(U_1) \cup  E_{\R}(U_2)$, say  $f \in E_{\R}(U_1)$ with vertices $v_1, v_2, v_3$, we may create copies of a Fano plane $F$ in $K+f$ as follows.  Fix four vertices $v_4, \ldots , v_7$ in  $U_2$ or one vertex $v_4$ in $U_1$ and three vertices $v_5, v_6, v_7$ in $U_2$, and choose six hyperedges $e_1, \ldots , e_6$ on the  seven vertices to form a Fano plane. The number of ways to do this is exactly
$$ 2 \binom{4}{2} \binom{|U_2|}{4} +  2 \binom{4}{2} (|U_1| - 3) \binom{|U_2|}{3}
\geq
u^{4},$$
where $u=\min\{|U_i| \colon i \in [2]\}$.

By Lemma~\ref{lem:sizes}  with (\ref{bound_over_E}) applied to $\mathcal{R}''$, we must have
\begin{eqnarray}\label{eq:ui}
u \geq \frac{m}{2} - 2 \sqrt{ 
  3\beta(\mathcal{R})^{\frac{1}{64}}}m.
\end{eqnarray}
Then,
\begin{eqnarray*}
|\mathcal{C}(f)|
&\geq& \left( \frac{m}{2} - 2 \sqrt{ 
 3\beta(\mathcal{R})^{\frac{1}{64}}} m \right)^4
\stackrel{(\beta(\mathcal{R})\leq \left(\frac{1}{432}\right)^{64})}{\geq} \left( \frac{m}{3} \right)^4.
\end{eqnarray*}
\end{proof}

By Claim~\ref{claim:copies} and (\ref{eq:aux}), we have
\begin{eqnarray}\label{bound_EN}
 e_{\R}(U_1) +  e_{\R}(U_2) 
\leq 243 \beta(\mathcal{R})^{\frac{1}{64} } m^3.
\end{eqnarray}
Thus,
\begin{eqnarray}\label{eq:fewcolors}
\sum_{j=1}^{6}e_j(\mathcal{R})
&\leq&  243 \beta(\mathcal{R})^{\frac{1}{64}} m^3 +  3 \beta(\mathcal{R})^{\frac{1}{64} }  m^3 
\leq 246 \beta(\mathcal{R})^{\frac{1}{64} }   m^3.
\end{eqnarray}
Consider all those terms $(\cV,\mathcal{R})$ in the sum (\ref{eqLL}) where the multicolored cluster hypergraph 
$\mathcal{R}$ satisfies $4h(r\eta) + 4r\eta \leq \beta(\mathcal{R}) \leq  X \leq 1/24^{64}$. Denoting the set of all these terms by $\P^{\ast \ast}$, we have
\begin{eqnarray} \label{eq:klm3}
&&
\sum_{(\cV,\mathcal{R}) \in \P^{\ast \ast}} 2^{h(r\eta)n^3}r^{r\eta n^3} \left(6^{\sum_{j=1}^{6}e_j(\mathcal{R})} r^{\ex(m,\Fano)-\beta(\mathcal{R}) m^3}\right)^{\left(\frac{n}{m}\right)^3} \nonumber
\\
&\leq&
M_0^n \cdot 2^{rM_0^3} \cdot 2^{h(r\eta)n^3} \cdot r^{r\eta n^3}
\left(
6^{246 \beta(\mathcal{R})^{\frac{1}{64}}   m^3}
r^{\ex(m,\Fano)-\beta(\mathcal{R}) m^3}\right)^{\left(\frac{n}{m}\right)^3} \nonumber
\\
&\stackrel{(n \gg 1)}{<}&
\frac{1}{2} r^{h(r\eta)n^3 + r\eta n^3} 
r^{\ex(n,\Fano)-\frac{\beta(\mathcal{R})}{2} n^{3}} 
\end{eqnarray}
provided that
\begin{equation}\label{eq_carlos1}
6^{246 \beta(\mathcal{R})^{\frac{1}{64}}} <r^{\frac{\beta(\mathcal{R})}{2}}.
\end{equation}

This holds if
$
6^{492}< r^{\beta (\mathcal{R})^{63/64}}$.
Since $\beta (\mathcal{R})\geq 4h(r\eta) + 4r\eta$, and by the lower bound in (\ref{eqLeta2}), 
namely
$$
 \left( \frac{\delta}{492} \right)^{64} \leq 4h(r\eta) + 4r\eta \leq \beta(\mathcal{R}),
$$ 
inequality~{\eqref{eq_carlos1}} holds if
$$
6^{492} < r^{\left(\frac{\delta}{492}\right)^{63}} \Longrightarrow 6^{\frac{492^{64}}{\delta^{63}}} < r,$$
which is precisely our assumption.

For $\beta (\mathcal{R})\geq 4h(r\eta) + 4r\eta$  inequality~(\ref{eq:klm3}) yields 
\begin{eqnarray} \label{eq:klm4}
&&
r^{- (h(r\eta) + r\eta) n^3} \cdot
r^{\ex(n,\Fano)}
<
\frac{1}{2} \cdot
r^{\ex(n,\Fano)}.
\label{eq:caseb1}
\end{eqnarray}
Combining (\ref{eq:casea1}) and (\ref{eq:klm4}), we have less than $r^{ex(n,F)}$ distinct $F^R$-free $r$-colorings when case (a) applies, so that it cannot happen.

\vspace{3 pt}

\noindent \emph{Case (b)} In this case, there is a partition $\cV$ of the vertex set of $H=(V,E)$ with $m$ classes that is associated with a multicolored cluster hypergraph $\mathcal{R}$ for which $\beta(\R) < 4h(r\eta) + 4r\eta$. Again, consider $\R'$ obtained from $\R$ by removing all hyperedges with less than six colors in their list of colors. As in case (a), Theorem~\ref{thm:kee1} tells us that $\R'$ admits a partition $\U = \{U_1,  U_2\}$ satisfying
$$e_{\mathcal{R}'}(U_1) + e_{\mathcal{R}'}(U_2) \leq 2 \beta(\mathcal{R})^{\frac{1}{64}} m^3.
 $$

Consider the partition $\mathcal{W} = \{W_1, W_2\}$ of the vertex set of $H$ where, for all $i\in[2]$,
$W_i = \bigcup_{j\in U_i}V_j.$ We want to find an upper bound on the cardinality of the set $E_{H}(W_1) \cup E_{H}(W_2)$ with respect to $\W$. Such  hyperedges are either one of the hyperedges counted by equations $\eqref{eqL1}, \eqref{eqL2}$ and $\eqref{eqL3}$ or are associated with a non-crossing hyperedge of $\mathcal{R}$ with respect to $\mathcal{U}$. Clearly, each non-crossing hyperedge of $\mathcal{R}$ with respect to $\mathcal{U}$ generates at most $(n/m)^3$ non-crossing hyperedges of $H$ with respect to $\mathcal{W}$. Using the upper bound on $ e_{\R}(U_1) +  e_{\R}(U_2)$ given in~\eqref{bound_EN}, we conclude that
\begin{eqnarray}\label{eq1}
e_{H}(W_1) +e_{H}(W_2)
&\leq&
r \eta n^3 +  \left(\frac{n}{m}\right)^{3} \cdot ( e_{\R}(U_1) +  e_{\R}(U_2)) \nonumber \\
&\leq&
\left(
r\eta
+ 
243 \beta({\mathcal R})^{\frac{1}{64}}  
\right)n^3
.
\end{eqnarray}
For $\beta({\mathcal R})<4h(r\eta) + 4r\eta$ expression~\eqref{eq1} is at most
\begin{eqnarray*}
&&
\left(
r\eta
+ 
 243  (4h(r\eta) + 4r\eta)^{\frac{1}{64}}
\right)n^3
\stackrel{(\ref{eqLeta}),(\ref{eqLeta2})}{\leq}
\left(
\frac{\delta}{4}
+ 
\frac{3\delta}{4}\right)n^3=\delta n^3
\end{eqnarray*}
as desired.
\end{proof}

\section{Proof of Theorem~\ref{thm:main:fano}} \label{sec:embedding}

In this section, we use the stability result of the previous section to prove that Theorem~\ref{thm:main:fano} holds. 

\begin{proof}[Proof of Theorem~\ref{thm:main:fano}]

Let $\gamma$, $\xi = \xi(\gamma)$, $\delta = \delta(\xi)$, $r_0 = r_0(\delta)$ and $r \geq  r_0$ be positive numbers satisfying
\begin{eqnarray} \label{eq:41a}
&&
\gamma\leq \frac{1}{1406}
\end{eqnarray}
\begin{eqnarray} \label{eq:41b}
&&
\xi\leq \frac{\gamma^3}{16}  
\end{eqnarray}
\begin{eqnarray} \label{eq:41c}
&&
\delta<\min\left\{\frac{1}{400^2}, \frac{\gamma^2}{4\cdot9^2}, \frac{\xi^3}{36}\right\}
\end{eqnarray}
\begin{eqnarray} \label{eq:41d}
&&
r>\max\{r_0, 21^{64}\},
\end{eqnarray}
hence we may have equality in~\eqref{eq:41a} and~\eqref{eq:41b} and fix
$$
\delta = \frac{1}{37 \cdot 16^3 \cdot 1406^9}, 
$$
and thus  
$$
r_0 = r_0(\delta)=6^{492^{64}\cdot (37\cdot 16^3\cdot 1406^9)^{63}} 
$$ comes from Lemma~\ref{lem:est:col:FK_fano}.


Let $n_1 = n_1(r,\delta)$ defined in Lemma~\ref{lem:est:col:FK_fano} for fixed $r$ and set $n_0 \geq n_1+3 \binom{n_1}{3}$. Consider $n\geq n_0$ large enough to ensure that (\ref{eq:nlarge1}), (\ref{eq:nlarge5}), (\ref{eq:nlarge2}), (\ref{eq:nlarge7}), (\ref{eq:nlarge6}) and (\ref{eq:nlarge3}) hold. Assume that we are given a hypergraph $H$ on $n\geq n_0$ vertices and with at least $r^{\ex(n,\Fano)+m}$ distinct $F^R$ colorings, for some $m\geq 0$.

\begin{claim}\label{claim:main}
If $H$ is a hypergraph with at least $r^{\ex(n,\Fano)+m}$ distinct $F^R$-free $r$-colorings, for some $m\geq 0$, and $H\neq B_n$, then there exists an induced sub-hypergraph $H^\prime$ on $n^\prime\geq n-3$ vertices and at least $r^{\ex(n^\prime,\Fano)+m+1}$ distinct $F^R$-free $r$-colorings.
\end{claim}

If Claim~\ref{claim:main} is true, we inductively arrive at some sub-hypergraph $H_{1}$ with $n'\geq n_1$ vertices that allows more than $r^{\binom{n'}{3}}$ feasible colorings, which is impossible and yields the desired contradiction.

To prove Claim~\ref{claim:main}, let $H=(V,E)\neq B_n$ be a $3$-uniform hypergraph on $n$ vertices and with at least $r^{\ex(n,\Fano)+m}$ feasible colorings, with $m\geq 0$. As $B_n$ can be colored arbitrarily without producing a rainbow Fano plane, we have that $|E(H)|\geq |E(B_n)|$. Let $\delta_1(H)$ be the minimum degree of $H$. If $\delta_1(H) < \delta_1(B_n)$, let $v$ be a vertex of minimum degree in $H$ and consider 
the sub-hypergraph $H^\prime = H-v$. Since
$$|E(B_{n-1})| = |E(B_n)| - \delta_1(B_n)\leq |E(B_n)|-\delta_1(H)-1,$$
we conclude that the number of $(F, P)$-free $r$-colorings of $H^\prime$ is at least
$$\frac{r^{|E(B_n)|+m}}{r^{\delta_1(H)}} = r^{|E(B_n)|+m-\delta_1(H)} \geq r^{|E(B_{n-1})|+m+1},$$
as desired. So let us assume that $\delta_1(H)\geq\delta_1(B_n)\geq 3n^2/8-n$. 

Consider a partition of $\V= \{X, Y\}$ of $V$ that minimizes $e_H(X) + e_H(Y)$. Let $E_C(H)$ and $E_N(H)$ denote the number of crossing hyperedges and non-crossing hyperedges in $H$ with respect to $\mathcal{V}$, respectively. By Lemma~\ref{lem:est:col:FK_fano}, we have $e_H(X) + e_H(Y)\leq\delta n^3$, and hence
\begin{eqnarray}\label{eqstar1}
|E(H)|\leq|E(B_n)|+\delta n^3.
\end{eqnarray}
It follows from $|E(H)|\geq |E(B_n)|$ that
\begin{eqnarray*}
&&|E_C(H)| \geq |E(B_n)|-\delta n^3,
\end{eqnarray*}
thus by Lemma~\ref{lem:sizes} we have
\begin{eqnarray} \label{eq:sizes2}
 n/2 - 2 \sqrt{\delta} n \leq \min\{|X|, |Y|\} \leq \max\{|X|, |Y|\} \leq n/2 + 2 \sqrt{\delta} n
 \end{eqnarray}
 
For a vertex $v$ of $H$, define its link graph $\L(v)$ with vertex set $V(H)-v$ and edge set $L(v) = \{\{u, w\}:\{v, u, w\}\in E(H)\}$. Every coloring of the hyperedges incident with $v$ naturally induces a coloring of $L(v)$, so that when we have a coloring of the hyperedges incident with $v$, we may view it as an edge-coloring of $\mathcal{L}(v)$. Given an $r$-coloring $\Delta$ of $H$, we say that a color $\alpha$ is \emph{abundant} with respect to a vertex $v$ and a class $Z\in\{X, Y\}$ when the set of hyperedges $e$ of color $\alpha$ such that $v \in e$ and $e-v \subseteq Z$ generate a matching of size at least ${\xi n}$ in  $\mathcal{L}(v)$, otherwise the color is called \emph{rare}.

Our argument is divided into two cases. First we shall assume that there exists a vertex $v$ with at least $\gamma n^2$ link edges in its ``own'' partition class. 

\vspace{5pt}

\noindent \textbf{Case 1.} $H$ has the property that there is a vertex $v$, without loss of generality $v \in Y$, such that $|L(v)\cap\binom{Y}{2}|\geq \gamma n^2$.

The minimality of the number $|E_N(H)|$ of non-crossing hyperedges in $H$ implies that $|L(v)\cap\binom{X}{2}|\geq \gamma n^2$, as otherwise we could move $v$ from $Y$ to $X$ to achieve a smaller number of non-crossing hyperedges. 

Let $\Delta$ be an $r$-coloring of $H$. For class $Z\in\{X, Y\}$, let $A_Z=A_Z(\Delta)$ be the set of abundant colors with respect to vertex $v$ and $Z$, and let $J_{Z}=J_Z(\Delta)$ be the set of edges in $L(v)\cap\binom{Z}{2}$ whose color is not in $A_X\cup A_Y$.

We split the set $\mathcal{C}$ of feasible colorings of $H$ into two disjoint classes $\mathcal{C}_1$ and $\mathcal{C}_2 = \mathcal{C}\setminus\mathcal{C}_1$, where $\mathcal{C}_1$ is the set of colorings for which either 
\begin{itemize}
\item[(a)] $|A_X\cup A_Y|\not\in\{1, 2\}$

or

\item[(b)] $|J_{Z}|\geq 4\sqrt{\xi}n^2$, for some class $Z\in\{X, Y\}$.
\end{itemize}
%
%

\begin{claim}\label{claim:triples}
In every coloring $\Delta \in \mathcal{C}_1$, there exists a matching $M$ in $\mathcal{L}(v)$ with the following property. There is a set $\mathcal{T}(v)$ of $3$-colored triples $(f_1, f_2, f_3) \in M^3$ such that:
\begin{itemize}
\item[(a)] $f_1\in L(v)\cap\binom{X}{2}, f_2\in L(v)\cap\binom{Y}{2}, f_3\in L(v)\cap\binom{Z}{2}$, with $Z\in\{X, Y\}$. 
\item[(b)] distinct triples $(f_1, f_2, f_3),(f_1', f_2', f_3') \in \mathcal{T}(v)$ satisfy $\{f_1, f_2, f_3\} \neq \{f_1', f_2', f_3'\}$.
\item[(c)] $\left| \mathcal{T}(v) \right| \geq  \xi^3n^3/9$.
 \end{itemize}
\end{claim}

\begin{proof}
Consider a coloring $\Delta \in \mathcal{C}_1$. We split the argument into three cases depending on the cardinalities of $A_X(\Delta)$ and $A_Y(\Delta)$.

\noindent \textit{Case 1)} $|A_X\cup A_Y|\geq 3$. Without loss of generality, assume that $|A_X|\geq |A_Y|$, which implies that $|A_X|\geq 2$. We consider two cases.

First assume that $A_Y\neq\emptyset$, so that there are three distinct colors $\alpha_1, \alpha_2\in A_X$ and $\alpha_3\in A_Y$. This implies that there exist vertex disjoint matchings $M_1$, $M_2$, $M_3$ in $\mathcal{L}(v)$, for which every edge in $M_i$ has color $\alpha_i$, $i=1, 2, 3$ and $|M_1| \geq \xi n/3$, $|M_2|\geq \xi n/3$, $|M_3|\geq \xi n$.  To see why this is true, note that, if we let $M_1$ be an arbitrary matching of size $\xi n/3$ whose elements have color $\alpha_1$, each edge in $M_1$ is incident with at most two edges in a maximum matching of color $\alpha_2$, so that at least $\xi n - 2\cdot\xi n/3 = \xi n/3$ edges in this matching are not incident with edges in $M_1$, which allows us to construct $M_2$. Set $M=M_1 \cup M_2 \cup M_3$. It is a matching with the property that all triples in $M_1 \times M_2 \times M_3$ are 3-colored. The number of triples is at least
$$
\left(\frac{\xi n}{3}\right)^2\xi n = \frac{\xi^3n^3}{9}.
$$

If $A_Y = \emptyset$, then consider colors $\alpha_1, \alpha_2\in A_X$ with vertex disjoint matchings $M_1$, $M_2$ having sizes $|M_1|\geq \xi n/3$, $|M_2|\geq \xi n/3$. As $A_Y=\emptyset$, the size of a maximum monochromatic matching of color $\alpha_1$ or $\alpha_2$ in $L(v) \cap \binom{Y}{2}$ is at most $\xi n$. Note that the number of edges of color $\alpha_1$ in $L(v) \cap \binom{Y}{2}$ is at most $2 \xi |Y| n$, as every edge in $L(v) \cap \binom{Y}{2}$ shares a vertex with at most $2(|Y|-2) \leq 2|Y|$ other edges in this set. This implies that the number of edges in $L(v)\cap\binom{Y}{2}$ with a rare color not in $\{\alpha_1, \alpha_2\}$ is at least
$$
\left|L(v)\cap\binom{Y}{2}\right| - 4\xi n |Y|
\geq
\gamma n^2 - 4\xi n^2
\stackrel{(\ref{eq:41b})}{\geq}
\frac{\gamma n^2}{2}.
$$
Then we may greedily find a matching $M_3$ in $L(v) \cap \binom{Y}{2}$ whose edges have colors in $[r]\setminus \{\alpha_1,\alpha_2\}$ of size at least
$$
\frac{\gamma n^2/2}{2|Y|} \stackrel{|Y| \leq n}{\geq} \frac{\gamma n}{4}.
$$
Set $M=M_1 \cup M_2 \cup M_3$. It is a matching with the property that all triples in $M_1 \times M_2 \times M_3$ are 3-colored. The number of triples is at least
$$
\left(\frac{\xi n}{3}\right)^2\cdot\frac{\gamma n}{4}
\stackrel{\eqref{eq:41a},(\ref{eq:41b})}{\geq}
\xi^3n^3
\geq
\frac{\xi^3n^3}{9}.
$$

\noindent \textit{Case 2)} $|A_X\cup A_Y|=0$. In this case, for every color $\alpha \in[r]$, the maximum size of a matching in $X$ or $Y$ of color $\alpha$ is at most $\xi n$. With a greedy construction we obtain that both $L(v)\cap\binom{X}{2}$ and $L(v)\cap\binom{Y}{2}$ contain matchings of size at least
$$\min\left\{\frac{|L(v)\cap\binom{X}{2}|}{2|X|}, \frac{|L(v)\cap\binom{Y}{2}|\}}{2|Y|}\right\}
\stackrel{|X|,|Y| \leq n}{\geq} \frac{\gamma n}{2}.$$

Fix matchings $M_X$ in $X$ and $M_Y$ in $Y$ for which
$$
\frac{\gamma n}{2}
\leq
\min\{|M_X|, |M_Y|\}
\leq
\max\{|M_X|, |M_Y|\}
\leq
\frac{n}{2}. 
$$
Clearly, $M=M_X \cup M_Y$ is a matching in $\mathcal{L}(v)$, which implies $|M_X| + |M_Y| \leq n/2$. For every color $\alpha \in[r]$, let $c_\alpha$ be the number of edges assuming color $\alpha$ in $M_X$ and $d_\alpha$ be the number of edges assuming color $\alpha$ in $M_Y$. Note that $c_\alpha, d_\alpha\leq\xi n$ and $\sum_{\alpha=1}^r c_\alpha\leq n/2 $, $\sum_{\alpha=1}^r d_\alpha\leq n/2$ and $\sum_{\alpha=1}^r (c_\alpha + d_\alpha) \leq n/2$.

The number of choices of triples $(f_1, f_2, f_3)$ with at most two colors where $f_1\in M_X$, $f_2\in M_Y$ and $f_3\in M_X\cup M_Y$, is at most
\begin{eqnarray*}
&&
\sum_{\alpha=1}^r c_\alpha d_\alpha \left(|M_X|+|M_Y|\right) +
\sum_{\alpha=1}^r c_\alpha |M_Y| (c_\alpha + d_\alpha) +
\sum_{\alpha=1}^r |M_X| d_\alpha (c_\alpha + d_\alpha)
\\
&\stackrel{|M_X| + |M_Y| \leq \frac{n}{2}}{\leq}&
\frac{n}{2}\sum_{\alpha=1}^r c_\alpha d_\alpha +
\frac{n}{2}\sum_{\alpha=1}^r c_\alpha (c_\alpha + d_\alpha) +
\frac{n}{2}\sum_{\alpha=1}^r d_\alpha (c_\alpha + d_\alpha)
\\
&\stackrel{(c_\alpha, d_\alpha\leq\xi n)}{\leq}&
\frac{\xi n^2}{2}\sum_{\alpha=1}^r c_\alpha +
\frac{\xi n^2}{2}\sum_{\alpha=1}^r (c_\alpha + d_\alpha) +
\frac{\xi n^2}{2}\sum_{\alpha=1}^r (c_\alpha + d_\alpha)
\\
&\leq&
\frac{\xi n^3}{4} +
\frac{\xi n^3}{4} +
\frac{\xi n^3}{4}
<
\xi n^3.
\end{eqnarray*}

As a consequence, the number of choices of $3$-colored triples $(f_1, f_2, f_3)$ is at least
$$
\left(\frac{\gamma n}{2}\right)^3 - \xi n^3
=
\frac{\gamma^3 n^3}{8}-\xi n^3\stackrel{(\ref{eq:41b})}{\geq}\frac{\gamma^3n^3}{16}
\stackrel{(\ref{eq:41b})}{\geq}
\frac{2\xi^3n^3}{3}
.$$
Each set $\{f_1,f_2,f_3\}$ may appear in at most $3!=6$ triples, so that we may define a set $\mathcal{T}(v)$ as in the statement of the claim with size at least $\xi^3n^3/9$, as required.

\noindent \textit{Case 3)} $|A_X\cup A_Y|\in\{1, 2\}$, and for some $Z\in\{X, Y\}$, $|J_{Z}|\geq 4\sqrt{\xi}n^2$. Without loss of generality we assume $Z=Y$, so that $|J_{Y}|\geq 4\sqrt{\xi}n^2$. First we look into the case where $A_X \neq \emptyset$. Then there exists a matching $M'$ in $J_{Y}$ of size at least
$$
\frac{4\sqrt{\xi}n^2}{2|Y|}
\stackrel{(\ref{eq:sizes2})}{\geq}   
\frac{4\sqrt{\xi}n^2}{\frac{3n}{2}}
\geq
2\sqrt{\xi}n+1,
$$
and hence the number of distinct pairs $\{f_2,f_3\}$ in $M'$  is at least
$$
\binom{2\sqrt{\xi}n+1}{2}
\geq
\frac{(2\sqrt{\xi}n)^2}{2}
=
2\xi n^2.
$$
For every color $\alpha \in[r]$, let $d_\alpha$ be the number of edges assuming color $\alpha$ in $M'$. Recall that all hyperedges in $J_Y$ are assigned colors in $[r] \setminus (A_X \cup A_Y)$. Since the number of pairs in $M' \subseteq J_{Y}$  with the same color is at most
$$
\sum_{\alpha=1}^{r}d_\alpha^2
\leq
\xi n\sum_{\alpha=1}^{r}d_\alpha
\leq
\frac{\xi n^2}{2},
$$
we have that the number of pairs $\{f_2,f_3\}$ as above such that $f_2$ and $f_3$ have different colors is at least
\begin{equation}\label{aux_eq}
2\xi n^2 - \frac{\xi n^2}{2}
\geq
\xi n^2.
\end{equation}
Let $J$ denote this set of pairs.

Since there is a color $\alpha\in A_X$, there exists a matching $M''$ of size $|M''|\geq \xi n$ in $X$ for which every edge in $M''$ assumes color $\alpha$. Clearly, $M=M'\cup M''$ is a matching. Form a set $\mathcal{T}(v) \subseteq M'' \times M'^2$ of triples given by one edge from $M''$ and a pair of edges in $J$, so that each triple is 3-colored. The cardinality of $\mathcal{T}(v)$ is at least
$$
\xi n \cdot \xi n^2
=
\xi^{2} n^3
\geq
\frac{\xi^3n^3}{9}.
$$
If $A_X = \emptyset$, then $|A_Y|\geq 1$. Let $\alpha\in A_Y$, and let $M'$ be a matching in $L(v)\cap \binom{Y}{2}$ of size $|M'|\geq \xi n$ for which every edge in $M'$ assumes color $\alpha$. Since $A_X = \emptyset$, we know that there are at least
$$
\left|L(v)\cap\binom{X}{2}\right| - \xi n\cdot2|X|
\geq
\gamma n^2 - 2\xi n^2
\stackrel{(\ref{eq:41b})}{\geq}
4\sqrt{\xi}n^2
$$
edges in $L(v)\cap\binom{X}{2}$ assuming a rare color different from $\alpha$. Thus, with calculations as the ones leading to~\eqref{aux_eq}, we obtain a set $J$ of at least $\xi n^2$ mutually vertex disjoint pairs of edges contained in a matching $M''$ in $X$ in different colors, which are different from $\alpha$. Set $M=M'\cup M''$ and define $\mathcal{T}(v)$ by forming triples with a pair of edges in $J$ and one edge in $M''$. The set $\mathcal{T}(v)$ contains at least
$$
\xi n^2\cdot\xi n \geq \frac{\xi^3n^3}{9}
$$
$3$-colored triples $(f_1, f_2, f_3)$.
\end{proof}

We wish to prove that $|\mathcal{C}_1|\leq r^{|E(B_n)|-1}$. For any triple $(f_1, f_2, f_3)$ in $\mathcal{T}(v)$, let $t_1, t_2, t_3, t_4\in\binom{V}{3}$ be four $3$-element sets (not necessarily hyperedges from $H$) such that $\{\{v\}\cup f_i: i=1, 2, 3\}\cup\{t_1, t_2, t_3, t_4\}$ forms a Fano plane. Note that each of the $3$-element sets $t_1, t_2, t_3, t_4$ contains precisely one vertex from each $f_i$. (In fact, there are two different sets of four $3$-element sets $t_1, t_2, t_3, t_4$ with this property for any given $f_1, f_2, f_3$ and we just fix one of those two sets arbitrarily.) Furthermore, note that for two different choices of $f_1, f_2, f_3$ and $f_1^\prime, f_2^\prime, f_3^\prime$, there is at least one $i \in \{1,2,3\}$ with $f_i \cap (f_1^\prime \cup f_2^\prime \cup f_3^\prime) = \emptyset$. Therefore the corresponding sets $\{t_1, t_2, t_3, t_4\}, \{t_1^\prime, t_2^\prime, t_3^\prime, t_4^\prime\}$ are disjoint. 

Fix any $r$-coloring $\Delta \in \mathcal{C}_1$. By Claim~\ref{claim:triples}, the coloring of $\mathcal{L}(v)$ induced by this coloring leads to at least $\xi^3n^3/9$ distinct $3$-colored triples in $\mathcal{T}(v)$. Fix one of these $3$-colored triples $(f_1, f_2, f_3)$. Since $\{v\}\cup f_i$ are colored in different colors, either one of the $3$-element sets $t_i$ must be missing from $H$, or altogether at most six colors are assigned to the seven edges $\{f_1, f_2, f_3, t_1, t_2, t_3, t_4\}$, because there is no rainbow Fano plane in $H$. This leads to at most
$$Q = 3 \cdot 4 \cdot r^3 + \binom{4}{2} r^3 = 18r^3$$
ways to extend the coloring of $\{f_1,f_2,f_3\}$ to $t_1, t_2, t_3, t_4$, because one of these four hyperedges may be assigned one of the colors used for $f_1,f_2,f_3$ or two of $\{t_1,t_2,t_3,t_4\}$ use the same new color.

For each triple $\{f_1,f_2,f_3\}$ in $\mathcal{T}(v)$, we choose a $4$-element set $\{t_1, t_2, t_3, t_4\}$ of crossing hyperedges in $H$ with respect to partition $\mathcal{V}$ that form a Fano plane with the triple $(\{v\} \cup f_1, \{v\} \cup f_2, \{v\} \cup f_3)$, if such a set exists (recall that there may be at most two). Let $T=T(v)$ be the family of pairs obtained in this way. 
We wish to find a lower bound on the cardinality of $T$. Every time a triple $\{f_1,f_2,f_3\}$ cannot be paired with a $4$-element set, it means that some crossing hyperedge must be missing from $H$. Moreover, for any two triples $\{f_1,f_2,f_3\}$ and $\{f'_1,f'_2,f'_3\}$ that cannot be paired with $4$-element sets, the missing hyperedges are different. Recall that the number of non-crossing hyperedges in $H$ with respect to $\mathcal{V}$ is at most 
$\delta n^3$ 
and that $H$ has at least $\ex(n,\Fano)$ hyperedges. This means that the number of triples that would be crossing hyperedges with respect to $\mathcal{V}$, but that do not lie in $H$, is at most 
$$\delta n^3 \stackrel{(\ref{eq:41c})}{\leq} \frac{\xi^3 n^3}{36}.$$
We conclude that
$$
|T|
\geq
\frac{\xi^3n^3}{9} - \delta n^3 \stackrel{(\ref{eq:41c})}{\geq}
\frac{\xi^3n^3}{9} - \frac{\xi^3n^3}{36}
=
\frac{\xi^3n^3}{12}.$$

To obtain an upper bound on $|\mathcal{C}_1|$, we have at most $r^{|L(v)|}$ ways to color hyperedges containing vertex $v$ and there are at most $Q^{|T|}$ ways to color the Fano planes in $H$ that extend $3$-colored triples $(f_1, f_2, f_3)$ in $\mathcal{T}(v)$, and finally at most $r^{|E(H)|-4|T|-|L(v)|}$ ways to color the remaining hyperedges of $H$. Therefore,
\begin{eqnarray} \label{eq:101x}
|\mathcal{C}_1|
\leq
Q^{|T|}r^{|E(H)|-4|T|}
&\leq&
(18r^3)^{|T|}r^{|E(H)|-4|T|}
.
\end{eqnarray}
The right-hand side of (\ref{eq:101x}) increases as $|T|$ decreases. Since $|T|\geq\xi^3n^3/12$, we have
\begin{eqnarray}
|\mathcal{C}_1|
&\leq&
(18r^3)^{\frac{\xi^3 n^3}{12}}r^{|E(H)|-4 \frac{\xi^3 n^3}{12}}
\nonumber\\
&\stackrel{(\ref{eqstar1})}{\leq}&
18^{\frac{\xi^3 n^3}{12}}r^{|E(B_n)|+\delta n^3- \frac{\xi^3 n^3}{12}}
\nonumber\\
&\stackrel{(\ref{eq:41c})}{\leq}&
\left(r^{\log_r18}\right)^{\frac{\xi^3 n^3}{12}}r^{|E(B_n)| - \frac{\xi^3 n^3}{18}}\nonumber\\
&=&
r^{|E(B_n)|+\frac{\xi^3}{12}\left(\log_r18-\frac{2}{3}\right)n^3}\nonumber\\
&\stackrel{(\ref{eq:41d})}{\leq}& r^{|E(B_n)|-1}.\label{eq:nlarge1}
\end{eqnarray}

We recall that $\mathcal{C}_2 = \mathcal{C}\setminus\mathcal{C}_1$ is the class of colorings $\Delta$ satisfying $|A_X(\Delta)\cup A_Y(\Delta)|\in\{1, 2\}$ and $|J_Z|<4\sqrt{\xi}n^2$ for every class $Z\in\{X, Y\}$. By our bound on $|\mathcal{C}_1|$, we derive
\begin{eqnarray}\label{eqstar3}
|\mathcal{C}_2| = |\mathcal{C}|-|\mathcal{C}_1| &\stackrel{(\ref{eq:nlarge1})}{\geq}& r^{|E(B_n)|+m}-r^{|E(B_n)|-1}
\geq
r^{|E(B_n)|+m-1}.
\end{eqnarray}

Next we estimate the number of colorings of the set of hyperedges incident to vertex $v$ that can be extended to a coloring in $\mathcal{C}_2$.

By definition of $\mathcal{C}_2$, in each class, there are at most $4\sqrt{\xi}n^2$ edges assuming colors in $[r] \setminus (A_X \cup A_Y)$. To count the number of colorings that can be extended to a coloring from $\mathcal{C}_2$, we first choose at most two colors for $A_X \cup A_Y$, which may be done in at most $r^2$ ways. For each class, there are at most $\binom{n^2}{4\sqrt{\xi}n^2}$ ways to choose hyperedges assuming colors in $[r] \setminus (A_X \cup A_Y)$, which may be colored in at most $r^{4\sqrt{\xi}n^2}$ ways. There are at most $r^{|X||Y|}$ ways to color hyperedges containing $v$ and one vertex from each class. Finally, each of the other hyperedges containing $v$ must assume a color in $A_X \cup A_Y$, so that we have at most $2^{|L(v)|-2\cdot4\sqrt{\xi}n^2-|X||Y|}$ ways to color the remaining hyperedges. We conclude that, for $n$ sufficiently large, the number of colorings of the set of hyperedges incident with vertex $v$ that can be extended to a coloring in $\mathcal{C}_2$ is at most 
\begin{eqnarray}
&&
r^2
\binom{n^2}{4\sqrt{\xi}n^2}^2
r^{8\sqrt{\xi}n^2}
r^{|X||Y|}
2^{|L(v)|-8\sqrt{\xi}n^2-|X||Y|}
\nonumber\\
&\stackrel{((\ref{eq:entropy1}),|L(v)|\leq n^2)}{\leq}&
2^{2h(4\sqrt{\xi})n^2+n^2-8\sqrt{\xi}n^2-|X||Y|}
r^{8\sqrt{\xi}n^2+|X||Y|+2}
\nonumber\\
&\stackrel{(|X||Y|\leq n^2/4)}{\leq}&
r^{(\log_r2)(2h(4\sqrt{\xi})n^2+n^2-8\sqrt{\xi}n^2)+8\sqrt{\xi}n^2+(1-\log_r2)\frac{n^2}{4}+2}
\nonumber\\
&=&
r^{\left[
(\log_r2)\left(2h(4\sqrt{\xi})-8\sqrt{\xi}+ \frac{3}{4}\right)+8\sqrt{\xi}+\frac{1}{4}
\right]n^2+2}
\nonumber\\
&\stackrel{(*)}{\leq}&
r^{\frac{26}{100}n^2}.
\label{eq:nlarge5}
\end{eqnarray}
Inequality (*) can be seen as follows.
We claim that
\begin{eqnarray}\label{eqc2geq}
(\log_r2)\left(2h(4\sqrt{\xi})-8\sqrt{\xi}+\frac{3}{4}\right)
+
8\sqrt{\xi}+\frac{1}{4}
<
\frac{26}{100},
\end{eqnarray}
because the derivative of the function 
$$f(x) =
(\log_r2)\left(2h(x)-2x +\frac{3}{4}\right)
+
2x+\frac{1}{4}$$
satisfies
$$f'(x) = \frac{1}{\ln r} \left( 2 \ln \left(\frac{1-x}{x}\right) - 2\ln 2\right) + 2,$$
hence is increasing for $0< x \leq 1/2$ and $r \geq 2$. As  $\xi\leq\gamma^3/16 \leq 1/(1406^3\cdot 16)$, to obtain (\ref{eqc2geq}), inserting $x=4 \sqrt{\xi} \leq(1/1406)^{3/2}$, gives with (\ref{eq:entropy2}) for $r\geq 21^{64}$:
\begin{eqnarray*}
f((1/1406)^{3/2})
&<&
(\log_r 2) \left(2\cdot0.001 - 2\left(\frac{1}{1406} \right)^{\frac{3}{2}} +\frac{3}{4}\right) +  2\left(\frac{1}{1406} \right)^{\frac{3}{2}} +\frac{1}{4},
\\
&<&
(\log_r 2)0.76 + 0.251 < \frac{26}{100}.
\end{eqnarray*}

Setting $H^\prime = H-v$, we obtain that the number of feasible colorings of $H^\prime$ is at least
\begin{eqnarray}
&&
\frac{|\mathcal{C}_2|}{r^{\frac{26}{100}n^2}}
\geq
\frac{r^{|E(B_n)|+m-1}}{r^{\frac{26}{100}n^2}}
=
r^{|E(B_n)|+m-1-\frac{26}{100}n^2}.
\nonumber\\
&\stackrel{(n \gg 1)}{\geq}&
r^{|E(B_n)|+m-1 - \frac{3n^2}{8} + n - \frac{5}{8} + 2} \nonumber \\
&\stackrel{(\ref{eq|E(B_n)|})}{\geq}& r^{|E(B_{n-1})|+m+1},\label{eq:nlarge2}
\end{eqnarray}
which proves Claim~\ref{claim:main} for hypergraphs $H$ satisfying the assumptions of Case~1.


\vspace{15pt}

\noindent \textbf{Case 2.} $H$ has the property that for every class $Z\in\{X, Y\}$ and every vertex $v\in Z$, the inequality $|L(v)\cap\binom{Z}{2}|\leq\gamma n^2$ holds.

Since $H\neq B_n$, there exists a hyperedge $e = \{v_1, v_2, v_3\}$ with all vertices in the same class, say $e\subseteq Y$. Let $\L$ be the graph with vertex set $X$ and edge set $L = \bigcap_{i=1}^{3}L(v_i)\cap\binom{X}{2}$. In particular, for any $f \in \binom{X}{2}$, $f$ lies in $L$ if and only if $f \in L(v_i)$ for every $i \in [3]$. As $|L(v_i)\cap\binom{Y}{2}|\leq\gamma n^2$ and $\delta_1(H)\geq\delta_1(B_n)\geq 3n^2/8-n$ we have 
$$
\left|L(v_i)\cap\binom{X}{2}\right|
\geq
\frac{3}{8} n^2-n-\gamma n^2 - |X||Y|
.$$

This implies that
$$
\left|\binom{X}{2} \setminus L(v_i) \right| \leq \binom{|X|}{2} - \frac{3}{8} n^2+\gamma n^2 + |X||Y|+n,
$$
so that 
\begin{eqnarray*}\label{eqdelta1}
|L|
&=&\binom{|X|}{2}-\left|\bigcup_{i=1}^3 \binom{X}{2} \setminus L(v_i) \right|\\
&\geq& \binom{|X|}{2}-3\left(\binom{|X|}{2} - \frac{3}{8} n^2+\gamma n^2 + |X||Y|+n\right)\\
&=&
\frac{9}{8}n^2 - 2\binom{|X|}{2} - 3|X||Y| - 3\gamma n^2-3n\\
&\geq&
\frac{9}{8}n^2 - |X|^2 - 3|X||Y| - 3\gamma n^2-3n
.
\end{eqnarray*}

\begin{claim}\label{claim:K_4}
There are at least
$$\frac{1}{6}
\left( 
\frac{2-240\gamma}{80}
\right)n^2$$
mutually edge-disjoint copies of the complete graph $K_4$ in $\mathcal{L}$.
\end{claim}

\begin{proof}
By Tur\'{a}n's theorem~\cite{turan}, a graph with $|X|$ vertices and more than
$|X|^2/3$
edges contains a $K_4$. Because of this, if $|L|>  |X|^2/3$, then there is a copy of $K_4$ in $\mathcal{L}$. Removing the six edges of this copy from $\L$, provided that $|L|-6 \geq |X|^2/3$, we may find another copy of $K_4$ that is edge-disjoint from the first one. Repeating this argument, the number of such copies of $K_4$ that we find is at least
\begin{eqnarray*}
&&\frac{1}{6}\left(|L|-\frac{|X|^2}{3}\right)
\geq
\frac{1}{6}
\left(
\frac{9}{8}n^2 - \frac{4}{3} |X|^2 - 3|X||Y| - 3\gamma n^2 -3n
\right).
\end{eqnarray*}

Since $|X|+|Y|=n$, and (\ref{eq:sizes2}) holds, we have that $4|X|^2/3 + 3|X||Y|$ is, without loss of generality, maximum for $|X| = n/2 + 2 \sqrt{\delta} n$ and $|Y | =  n/2 - 2 \sqrt{\delta} n$, i.e.,
\begin{eqnarray}
&&
\frac{4}{3}|X|^2+3|X||Y| + 3n
\nonumber\\
&\leq&
\label{eq:nlarge7}
\left(\frac{4}{3}\left(\frac{1}{2}+2\sqrt{\delta}\right)^2 + 3\left(\frac{1}{4}-4\delta\right)\right)n^2
+ 3n
\stackrel{(\ref{eq:41c})}{\leq}
\frac{11}{10}n^2.
\end{eqnarray}
Thus, we have at least
\begin{eqnarray*}
\frac{1}{6}\left(|L|-\frac{|X|^2}{3}\right)
&\geq&
\frac{1}{6}
\left(
\frac{9}{8} - 
\frac{11}{10} - 3\gamma
\right)n^2
=
\frac{1}{6}
\left( 
\frac{2-240\gamma}{80}
\right)n^2
\end{eqnarray*}
mutually edge-disjoint copies of $K_4$ in $\mathcal{L}$.
\end{proof}

Let $K^1, \ldots, K^q$ be the mutually edge-disjoint copies of $K_4$ in $\mathcal{L}$ given by Claim~\ref{claim:K_4}, where
\begin{eqnarray}\label{eq4.4}
q
\geq
\frac{1}{6}
\left( 
\frac{2-240\gamma}{80}
\right)n^2.
\end{eqnarray}
Since $E(K^j)\subseteq L$ for every $j\in[q]$, every such $K^j$ forms a Fano plane together with the hyperedge $e$. Fixing a color for $e$, we can color the six hyperedges that correspond to the edges of every $K^j$ in less than  $6r^5 + \binom{6}{2}r \cdot r^4 = 21r^5$ ways.

Set $H^\prime = H- e $ (that is, the vertices in $e$ are deleted from $H$). Let $E_e$ denote the set of hyperedges of $H$ that contain at least one vertex from $e = \{v_1, v_2, v_3\}$. Obviously, $|E_e| \leq 3\gamma n^2 + 3\binom{|X|}{2} + 3|X||Y|$. From the assumption $|E(X)|+|E(Y)|\leq\delta n^3$ and the fact that $|E(H)|\geq |E(B_n)|$, it follows that, for $n$ sufficiently large,
\begin{eqnarray*}
|E_e|
&\leq&
3\binom{|X|}{2} + 3|X||Y| + 3\gamma n^2
\nonumber\\
&\stackrel{(\ref{eq:sizes2})}{\leq}&
\frac{9}{2}\left(\frac{1}{2}+2\sqrt{\delta}\right)^2n^2+3\gamma n^2
\nonumber\\
&\stackrel{(\ref{eq:41c})}{\leq}&
\frac{9}{8}n^2+4\gamma n^2  \nonumber \\
&\stackrel{(n \gg 1)}{\leq}& \delta_1(B_n)+\delta_1(B_{n-1})+\delta_1(B_{n-2})+5\gamma n^2
\\
&=&|E(B_n)|-|E(B_{n-3})|+5\gamma n^2
\nonumber,
\end{eqnarray*}
which implies that
\begin{eqnarray} \label{eq:nlarge6}
|E(B_n)|-|E_e|&\geq& |E(B_{n-3})|-5\gamma n^2.
\end{eqnarray}

We can color hyperedges in $E_e$ in at most
$$r^{|E_e|}\left(\frac{21r^5}{r^6}\right)^q = 21^qr^{|E_e|-q}$$
ways.

Consequently, for $n$ sufficiently large, the number of feasible colorings of $H^\prime$ is at least
\begin{eqnarray}
\frac{r^{|E(B_n)|+m}}{21^qr^{|E_e|-q}}
&=&
21^{-q}r^{|E(B_n)|+m-|E_e|+q}
\stackrel{(\ref{eq:nlarge6})}{\geq}
r^{|E(B_{n-3})|+m-5\gamma n^2+q(1-\log_r21)}\nonumber\\
&\stackrel{(\ref{eq4.4})}{\geq}&
r^{|E(B_{n-3})|+m-5\gamma n^2+\frac{1}{6}
\left( 
\frac{2-240\gamma}{80}
\right)n^2(1-\log_r21)}\nonumber\\
&\geq&
r^{|E(B_{n-3})|+m+1}\label{eq:nlarge3},
\end{eqnarray}
where the last inequality can be seen as follows.
From (\ref{eq:41a}) and (\ref{eq:41d}) we obtain that
\begin{eqnarray*}\label{eq4.2}
-5\gamma+\frac{1}{6}
\left( 
\frac{2-240\gamma}{80}
\right)
(1-\log_r21)
>
0, 
\end{eqnarray*}
because $\gamma\leq1/1406$ and $r>21^{64}$ implies that $(1-\log_r21)>63/64$.

This concludes Case~2, finishes the proof of Claim~\ref{claim:main} and consequently the proof of Theorem~\ref{thm:main:fano}.
\end{proof}

\section{Final remarks and open problems}\label{sec_final}

In this paper, we have shown that, for sufficiently large $r$ and $n$, the hypergraph $B_n$ is the unique $n$-vertex 3-uniform hypergraph admitting the largest number of $r$-edge colorings with no rainbow copy of the Fano plane. A natural question would be to ask for the best possible values of $r$ and $n$ for which this holds. With respect to $r$, it is clear that our result cannot possibly be extended to $r \leq 10$, as the complete $3$-uniform $n$-vertex hypergraph $K_n^{(3)}$ admits at least $\max\{r^{\binom{n}{3}},6^{\binom{n}{3}}\}$ distinct $r$-colorings in which at most six colors are used, and this is larger than $r^{\ex(n,\Fano)}$ for $r \leq 10$. On the other hand, we are convinced that the value of $r_0$ provided in Theorem~\ref{thm:main:fano} is far from optimal. To the best of our knowledge, it may even be that $r_0=11$. In fact, any improvement on the range and influence of $\delta$ in Theorem~\ref{thm:kee1} would immediately be translated into a better value of $r_0$ in Lemma~\ref{lem:est:col:FK_fano}. This lemma is the only significant hurdle for better bounds on $r_0$ using the current approach, that is, it is possible to adapt our proof of Theorem~\ref{thm:main:fano} to lesser values of $r$ such that Lemma~\ref{lem:est:col:FK_fano} holds.

Regarding the value of $n_0=n_0(r)$ given in the proof of Theorem~\ref{thm:main:fano}, we believe that it is by no means optimal and we have made no particular effort to optimize it. Indeed, our proof is based on the weak regularity lemma~\cite{chung91,fr92,KNRSembedding,steger90}, which requires very large values on $n_0$ in the worst case. In the graph case, better bounds have been obtained in~\cite{BalLi18?,han} using strategies based on the container method~\cite{BMS2015,ST2015} (see also~\cite{multicoloredcontainers} for a colored version).

We would also like to mention that our proof of Lemma~\ref{lem:est:col:FK_fano} applies in more general contexts, as we now describe. For fixed positive integers $\ell$ and $k$, let $\mathcal{I}_{\ell, k}$ be the set of nonnegative integral solutions to  $x_1+\cdots+x_{\ell}=k$.
\begin{definition}\label{defmultipartitehypergraph}
For integers $n$, $\ell$ and $k$, a vector $I = (x_1, \ldots, x_{\ell})\in\mathcal{I}_{\ell, k}$ and a partition $\V = \{V_1, \ldots, V_{\ell}\}$ of $[n]$, let $H_{I, \V}(n)$ be the $k$-uniform hypergraph with vertex set $[n]$ and edge set given by all $k$-element subsets $e\subseteq [n]$ such that there is a permutation $\pi$ of $[\ell]$ for which $\left(|e\cap V_{\pi(1)}|, \ldots, |e\cap V_{\pi(\ell)}|\right) = I$. The hypergraph $H_{I, \V}(n)$ is called the \emph{complete multipartite hypergraph} with respect to $I$ and $\V$. We say that $I$ and $\V$ are the \emph{intersection vector} and the \emph{partition} of $H_{I, \V}(n)$, respectively. Moreover, if $H=K_n^{(k)}$ is the complete $k$-uniform hypergraph on $[n]$, we say that $B_{I, \V}(H)=E(H) \setminus E(H_{I, \V}(n))$ is the set of \emph{bad hyperedges} of $H$ with respect to $\V$ and $I$.
\end{definition}
For example, when $\ell=2$, $k=3$, $I=(2,1)$ and $\V=\{V_1,V_2\}$ is a balanced partition of $[n]$, we have $H_{I, \V}(n)=B_n$. Moreover, if $H$ is any $3$-uniform hypergraph on $[n]$, then $B_{I, \V}(H)$ is the set of hyperedges of $H$ that are entirely contained in $V_1$ or in $V_2$.

\begin{definition}[Well-behaved multipartite extremal hypergraph]\label{defturanextremal}
A $k$-uniform hypergraph $F$ has a \emph{well-behaved multipartite extremal hypergraph} if there exist positive integers $\ell$ and $n_0$, a constant $\xi>0$, and an intersection vector $I\in\mathcal{I}_{\ell, k}$ satisfying the following properties.
\begin{itemize}
\item[(a)] For every positive integer $n$ and for every partition $\V = \{V_1, \ldots, V_{\ell}\}$ of $[n]$, the hypergraph $H_{I, \V}(n)$ is $F$-free;

\item[(b)] For any integer $n\geq n_0$, there is a partition $\V = \{V_1, \ldots, V_{\ell}\}$ of $[n]$ for which $\ex(n, F) = |E(H_{I, \V}(n))|$. Moreover, this partition satisfies $|V_i| \geq \xi n$ for all $i \in [\ell]$.

\item[(c)] For every $\delta>0$, there exist an integer $n_1 \geq n_0$ and a constant $\epsilon_s = \epsilon_s(\delta)$ such that every $F$-free hypergraph $H=(V, E)$ on $n\geq n_1$ vertices with at least $\ex(n, F)-\epsilon_s n^k$ hyperedges admits a partition $\V = \{V_1, \ldots, V_{\ell}\}$ of $V$ such that
$$|B_{I, \V}(H)| \leq \delta n^k.$$
\end{itemize}
\end{definition}
To the best of our knowledge, the structural description of a well-behaved multipartite extremal hypergraph applies to all linear hypergraphs $F$ whose extremal hypergraph is known and is dense. Examples include expanded complete graphs (see Pikhurko~\cite{PH_ell+1^k} and Mubayi~\cite{Mub06} for the extremal and the stability result, respectively), fan hypergraphs (see Mubayi and Pikhurko~\cite{MPfan}) and for more general classes generalizing these instances (see Brandt, Irwin and Jiang~\cite{BIJ18} and Norin and Yepremyan~\cite{NY2018}).

A few changes in our proof of Lemma~\ref{lem:est:col:FK_fano} yield the following stability result. A proof of this result may be found in the first author's doctoral thesis~\cite{contiero_thesis}.
\begin{theorem}\label{lemmamain2}
Let $k\geq 2$ be an integer. For every $\delta>0$ and every linear $k$-uniform hypergraph $F$ that has a well-behaved multipartite extremal hypergraph with intersection vector $I\in\I_{\ell, k}$, there is $r_0 = r_0(\delta,F)$ with the following property. For all integers $r\geq r_0$ there exists $n_0 = n_0(r)$ such that, if $n\geq n_0$ and $H = (V, E)$ is an $n$-vertex $k$-uniform hypergraph with at least $r^{\ex(n, F)}$ distinct $(F, R)$-free $r$-colorings, where $R$ is the rainbow pattern of $F$, then there is a partition $\V = \{V_1, \ldots, V_\ell\}$ of $V$ for which
$$|B_{I, \V}(H)|\leq \delta n^k.$$
\end{theorem}

On the other hand, the part of our proof that uses the stability of Lemma~\ref{lem:est:col:FK_fano} to show that $B_n$ is the single $n$-vertex $(r,F^R)$-extremal graph for $r \geq r_0$ and sufficiently large $n$ (see Claim~\ref{claim:main}) uses ad-hoc arguments that rely heavily on the structure of the Fano plane and on $B_n$. We have not been able to generalize it to other hypergraphs $F$. However, we believe that the following general statement is true.
\begin{conjecture}\label{conj1}
Given an integer $k\geq 2$ and a linear $k$-uniform hypergraph $F$ such that $\ex(n,F)=\Omega(n^k)$, there exists $r_0$ with the following property. For every $r \geq r_0$, there is $n_0$ such that any $n$-vertex $k$-uniform hypergraph $H$, where $n \geq n_0$, satisfies
\begin{equation}\label{eq_conj}
c_{r,(F,P)}(H) \leq r^{\ex(n,F)},
\end{equation}
where $P$ is the rainbow pattern of $F$. Moreover, equality holds in~\eqref{eq_conj} for $n \geq n_0$ if and only if $H$ is $F$-extremal.
\end{conjecture}
We should point out that we do not expect the statement of Conjecture~\ref{conj1} to hold for $k$-uniform hypergraphs $F$ whose $F$-extremal hypergraph is sparse (see Conjecture~\ref{conj2} below). Moreover, as done in~\cite[Remark 4.1]{rainbow_complete}, one may show that there exist $k$-uniform hypergraphs $F$ such that the statement of Conjecture~\ref{conj1} does not hold for any non-rainbow pattern of $F$. 

For more general patterns, we also deem the following to be true, strengthening~\cite[Theorem 1.1]{rainbow_complete}. Here, we let $K_n^{(k)}$ denote the complete $k$-uniform hypergraph on $n$ vertices.
\begin{conjecture}\label{conj2}
Fix integers $r,k\geq 2$ and a linear $k$-uniform hypergraph $F$ such that $\ex(n,F)=o(n^k)$. Let $P$ be any pattern of $F$ on $t \geq 3$ classes. Then there exists $n_0$ such that 
 any $n$-vertex $k$-uniform hypergraph $H$, where $n \geq n_0$, satisfies
\begin{equation}\label{eq_conj2}
c_{r,(F,P)}(H) \leq c_{r,(F,P)}(K_n^{(k)}).
\end{equation}
 Moreover, equality holds in~\eqref{eq_conj2} for $n \geq n_0$ if and only if $H=K_n^{(k)}$.
\end{conjecture}

\end{document}